\documentclass[a4paper,11pt]{article}

\usepackage[utf8]{inputenc}
\usepackage[english]{babel}
\usepackage{enumitem}
\usepackage{color}
\usepackage{amsmath}
\usepackage{amssymb}
\usepackage{amsthm}
\usepackage{amsfonts}
\usepackage{epsfig}
\usepackage{graphicx}
\usepackage{cite}
\usepackage{epstopdf}
\usepackage{tabularx}

\usepackage{caption}
\newtheorem{theorem}{Theorem}[section]

\newtheorem{lemma}{Lemma}[section]

\newtheorem{remark}{Remark}[section]

\newtheorem{algorithm}{Algorithm}[section]

\newtheorem{example}{Example}[section]
\numberwithin{equation}{section}

\theoremstyle{definition}

\allowdisplaybreaks

\date{}

\begin{document}

\title{\textbf{Convergence analysis of the Halpern iteration with adaptive anchoring parameters}}
\author{
{\sc Songnian He$^{1,}${\thanks{email: songnianhe@163.com}}},\quad
{\sc Hong-Kun Xu$^{2,3,}${\thanks{Corresponding author. email: xuhk@hdu.edu.cn }}},\quad
{\sc Qiao-Li Dong$^{1,}${\thanks{email: dongql@lsec.cc.ac.cn}}},\quad
{\sc Na Mei$^{1,}${\thanks{email: meina18812553229@163.com}}}\\
{\footnotesize $^1$College of Science, Civil Aviation University of China, Tianjin 300300, China}\\
{\footnotesize $^2$School of Science, Hangzhou Dianzi University, Hangzhou, 310018, China }\\
{\footnotesize $^3$College of Mathematics and Information Science, Henan Normal University, Xinxiang, 453007, China}\\
}
\maketitle

\begin{abstract}
We propose an adaptive way to choose the anchoring parameters for the Halpern iteration to find a fixed point of a nonexpansive mapping in a real Hilbert space. We prove strong convergence of this adaptive Halpern iteration and obtain the rate of asymptotic regularity at least $O(1/k)$, where $k$ is the number of iterations.
Numerical experiments are also provided to show advantages and outperformance
of our adaptive Halpern algorithm over the standard Halpern algorithm.

\end{abstract}



\section{Introduction}

Let $\mathcal{H}$ be a real Hilbert space  with  inner product $\langle\cdot,\cdot\rangle$ and norm $\|\cdot\|$, respectively.
Recall that a mapping $T: \mathcal{H}\rightarrow \mathcal{H}$ is said to be nonexpansive if, for each $x,y\in \mathcal{H}$,
$$\|Tx-Ty\|\leq \|x-y\|.$$
We use $Fix(T)$ to denote the set of fixed points of $T$, that is, $Fix(T)=\{x\in \mathcal{H}\, |\, x=Tx\}$.
It is known that $Fix(T)$ is always convex and that $Fix(T)\not=\emptyset$ if and only if,
for each $x\in \mathcal{H}$, the sequence of trajectories of $T$ at $x$, $\{T^nx\}_{n=0}^\infty$, is bounded.

It is always an interesting topic to find a fixed point of a nonexpansive mapping.
This is nontrivial since unlike the case of contractions,
in the case of a nonexpansive mapping $T$, the sequence of the Picard iterates of $T$
at a point $x$, $\{T^nx\}_{n=0}^\infty$, may fail to converge (for instance, a rotation around the origin
in the plane $\mathbb{R}^2$).

The Halpern iteration  \cite{Halpern1967} was proposed by Halpern in 1967 in  a Hilbert space. This method generates,
 with an initial guess $x^0\in \mathcal{H}$ arbitrarily chosen,  a sequence $\{x^k\}_{k=0}^\infty$ by the iteration process:
 \begin{equation}\label{Halpern-1}
 x^k=\lambda_k u+(1-\lambda_k)Tx^{k-1},\qquad k=1,2,\cdots,
 \end{equation}
where $u$ is a fixed point in $\mathcal{H}$, referred to as anchor, and the parameters $\{\lambda_k\}_{k=1}^\infty$ are in $(0,1)$,
which will be referred to as anchoring parameters.

Halpern \cite{Halpern1967} discovered that in order that \eqref{Halpern-1} converges for every nonexpansive mapping $T$ with $Fix(T)\not=\emptyset$
and arbitrary anchor $u\in\mathcal{H}$, the following two conditions are necessary (but not sufficient):
\begin{itemize}
\item[(C1)] $\lim_{k\to\infty}\lambda_k=0$,
\item[(C2)] $\sum_{k=1}^\infty\lambda_k=\infty$.
\end{itemize}

To guarantee convergence of Halpern iteration \eqref{Halpern-1}, an additional condition must be satisfied.
Any one of the following conditions is such an additional sufficient condition:
\begin{itemize}
\item[(C3)] (Halpern \cite{Halpern1967}) $(\lambda_k)$ is acceptable: there exists $(k(i))$ such that (i) $k(i+1)\ge k(i)$,
(ii)  $\lim_{i\to\infty}\frac{\lambda_{i+k(i)}}{\lambda_i}=1$, (iii) $\lim_{i\to\infty} k(i)\lambda_i=\infty$;
\item[(C4)] (Lions \cite{Lio77})
$\lim_{k\to\infty}\frac{|\lambda_{k+1}-\lambda_k|}{\lambda_k^2}=0$ (e.g., $\lambda_k=\frac{1}{(k+1)^\alpha}$, $0<\alpha<1$);
\item[(C5)] (Wittmann \cite{Wittmann1992})
$\sum_{n=1}^\infty |\lambda_{k+1}-\lambda_k|<\infty$ (e.g., $\lambda_k=\frac{1}{(k+1)^\alpha}$, $0<\alpha\le 1$);
\item[(C6)] (Reich \cite{Rei94})  $(\lambda_k)$ is decreasing;
\item[(C7)]  (Xu \cite{Xu2002})
$\lim_{k\to\infty}\frac{|\lambda_{k+1}-\lambda_k|}{\lambda_k}=0$,  i.e., $\frac{\lambda_{k+1}}{\lambda_k}\to 1$
(e.g., $\lambda_k=\frac{1}{(k+1)^\alpha}$, $0<\alpha\le 1$).
\end{itemize}

An advantage of Halpern's iteration  \eqref{Halpern-1} over some other iterations (such as
the Krasnosel'ski--Mann iteration) is that it is always strongly convergent even in an infinite-dimensional Hilbert space
and moreover, the limit is identified as the metric projection of the anchor $u$ onto the fixed point set $Fix(T)$.

Some quantitative properties on the displacements  $\|x^k-Tx^k\|$
of the Halpern iteration \eqref{Halpern-1} in both Hilbert and Banach spaces have been studied by several researchers,
see \cite{LL2007,KU2011,KU2012}.
 He, et al. \cite{He2019} studied optimal parameters of the Halpern iteration and gave an adaptive selection method of the approximate
 optimal parameters.
For more detail, the reader is referred to the survey article \cite{Lop10}.

Two more major progresses have been achieved recently on the Halpern iteration \eqref{Halpern-1}. The first one is
successful applications in machine learning (generative adversarial networks (GANs), in particular) \cite{Dia20,Yoo21}
and other applied areas such as minimax problems (see, e.g., \cite{QiX21}).
The second one is the (tight) optimal rate of asymptotic regularity proved by Lieder \cite{Lieder2021} (see also \cite{Sab17})
in a Hilbert space:
\begin{equation}\label{Lieder}
\|x^k-Tx^{k}\|\leq \frac{2}{k+1}\|x^0-x^*\|,\quad k\geq 1,
\end{equation}
where $x^*$ is an arbitrary fixed point of $T$, and where one assumes $u=x^0$ and $\lambda_k=\frac{1}{k+1}$ for all $k\geq 1$.


It has been brought to our attention that the anchoring parameters in the conditions (C1)-(C7) and also in Lieder's
asymptotic regularity rate \eqref{Lieder} are chosen in an open loop way.

In general, the purpose of adaptively selecting the parameters of an algorithm is to speed up the convergence of the algorithm,
  though updating the parameters may require certain additional computing work slightly.
An adaptive parameter selection strategy is believed to be effective if it significantly improves the convergence speed of the algorithm with little extra computational effort.
  It is therefore wondered if the rate of asymptotic regularity
of Halpern's iteration  \eqref{Halpern-1} can be improved should the anchoring parameters are chosen in an adaptive way,
This is the main problem to be dealt in this paper.
 More precisely, suppose $x^{k-1}(k\geq 1)$ has already been obtained,
 we will introduce an adaptive way to choose the
anchoring parameters as $\lambda_k:=\frac{1}{\varphi_k+1}\,(k\geq 1)$,
where $\varphi_k$ is determined by an adaptive selection method (see \eqref{eq3.1} in Section \ref{Sec:3}), which is very different from
works mentioned above.  A motivation to choose $\varphi_k$ in such a way as given in \eqref{eq3.1} is to shrink the gap
between $\|Tx^{k-1}-Tx^k\|^2$ and $\|x^{k-1}-x^k\|^2$ since we have $\|Tx^{k-1}-Tx^k\|^2\leq\|x^{k-1}-x^k\|^2$ by nonexpansiveness of $T$.
It then follows from the definition of $x^k$ that (see the details of the derivation of \eqref{eq3.9})
\begin{equation}\label{Halpern-1.3}
\aligned
0&\le \|x^{k-1}-x^k\|^2-\|Tx^{k-1}-Tx^k\|^2\\
&=\frac{2}{\varphi_k}\langle x^k-Tx^k, x^0-x^k \rangle-\|x^k-Tx^k\|^2\\
&\quad +\|x^{k-1}-Tx^{k-1}\|^2-\frac{2}{\varphi_k+1}\langle x^{k-1}-Tx^{k-1}, x^0-Tx^{k-1}\rangle.
\endaligned
\end{equation}
Our strategy is to choose the parameter $\varphi_k$ such that (see also \eqref{eq3.10})
\begin{equation}\label{Halpern-1.4}
\|x^{k-1}-Tx^{k-1}\|^2=\frac{2}{\varphi_k+1}\langle x^{k-1}-Tx^{k-1}, x^0-Tx^{k-1} \rangle.
\end{equation}
[This then yields \eqref{eq3.1} in Section \ref{Sec:3}.]
Thus, from \eqref{Halpern-1.3} and \eqref{Halpern-1.4}, we get a basic estimate:
\begin{equation}\label{Halpern-1.5}
\|x^k-Tx^k\|^2\leq \frac{2}{\varphi_k}\langle x^k-Tx^k, x^0-x^k \rangle.
\end{equation}

Basing upon \eqref{Halpern-1.5}, we will prove the strong convergence of Halpern iteration \eqref{Halpern-1} under adaptively chosen
anchoring parameters and also discuss the rate of asymptotic regularity. An example shows that our rate of asymptotic regularity is better than \eqref{Lieder}. By \eqref{Halpern-1.4} (or \eqref{eq3.1}), we also find that the amount of work required to calculate the parameters $\varphi_k$ is little.

The organization of the paper is as follows. In the next section we will include some basic tools for proving weak and strong
convergence in a Hilbert space. In Section \ref{Sec:3} we prove the main convergence results on
the Halpern iteration  in which the anchoring parameters are selected in an adaptive way as briefly described above.
The main results include strong convergence of the method and the following rate of asymptotic regularity:
\begin{equation}\label{He-Xu-New}
\|x^k-Tx^{k}\|\leq \frac{2}{\varphi_k+1}\|x^0-x^*\|,\quad k\geq 1,
\end{equation}
where $x^*$ is an arbitrary fixed point of $T$.
Since we will prove that $\varphi_k\geq k$ for all $k\geq 1$, (\ref{He-Xu-New}) is, in general, an improvement of
Lieder's rate (\ref{Lieder}).  We will also demonstrate an example to illustrate that our rate \eqref{He-Xu-New}
is indeed better than (\ref{Lieder}) in certain circumstances.

In Section \ref{Sec:4} we briefly discuss the case where the adaptive anchoring parameters are summable.
In this case we find that the Halpern iterates converge, but not to the projection of the anchor $u$ onto the
fixed point set $Fix(T)$ of $T$; it is instead the projection of the anchor $u$ onto another closed convex subset.

Numerical experiments will be carried out in Section \ref{Sec:5} to show efficiency of our adaptive Halpern iteration.
Computing results show that our adaptive Halpern iteration outperforms the ordinary Halpern iteration (i.e., Halpern iteration with
anchoring parameters chosen by the usual open loop manner).

\section{Preliminaries}
Some  tools are listed in this section, which will be used in the proofs of our main results. The following notation will be used
throughout the rest of the paper.
\begin{itemize}
  \item [(i)]  $x^k\rightarrow x$  denotes the strong convergence of $(x^k)$ to $x$.
  \item [(ii)] $x^k \rightharpoonup  $ denotes the weak convergence of $(x^k)$ to $x$.
  \item [(iii)] $\omega_w(x^k) :=\{x\mid\exists \,\, {\rm subsequence}\,\,   \{x^{k_i}\}_{i=1}^\infty\subset\{x^k\}_{k=1}^\infty$
  such that $ x^{k_i} \rightharpoonup x\}$ denotes the $\omega$-weak limit point set of $ \{x^k\}_{k=1}^\infty$.
 \end{itemize}

The metric (nearest point) projection $ P_{C}$ from a real Hilbert space $\mathcal{H}$
onto a nonempty closed convex subset $C\subset \mathcal{H}$  is defined by
$$ P_{C}(x)=\arg\min \{\|x-y\|\,|\,~y\in C \},~x\in \mathcal{H}.$$
It is well-known that $P_{C}$ is nonexpansive and the following characteristic inequality holds.
\begin{lemma}{\rm\cite[Section 3]{GR1984}}
\label{lem22}
 Let $z\in\mathcal{H}$ and $u\in C$. Then
 $u=P_{C}z$ if and only if$$\langle z-u,v-u\rangle\leq 0,~\quad v\in C .$$
 \end{lemma}

\begin{lemma} {\rm (The demiclosedness principle for nonexpansive mappings \cite{GK1990}.)\label{lemmaGK}}
Let $C$ be a nonempty closed convex subset of a real Hilbert space $\mathcal{H}$ and let $T: C\rightarrow C$ be a nonexpansive mapping such that $Fix(T)\neq\emptyset$. If a sequence $\{x^k\}_{k=0}^\infty$ in $C$ is such that $x^k\rightharpoonup z$ and $\|x^k-Tx^k\|\rightarrow 0$, then $z=Tz$.
\end{lemma}

\begin{lemma} {\rm(\cite{LX2007})}
\label{lem-1}
 Let $D$ be a nonempty subset of $\mathcal{H}$. Let $\{u^{k}\}_{k=0}^\infty \subset \mathcal{H}$ satisfy the properties:
 \begin{itemize}
\item[{\rm(i)}] $\lim_{k\rightarrow\infty}\|u^k-u\|$ exists for each $u\in D$;
\item[{\rm(ii)}] $\omega_w(u^k)\subset D$.
\end{itemize}
Then $\{u^{k}\}_{k=0}^\infty$ converges weakly to a point in $D$.
\end{lemma}

\begin{lemma} {\rm(\cite{Polyak1987})}
\label{lem-2}
Assume that $\{a_{k}\}_{k=0}^\infty$, $\{\lambda_{k}\}_{k=0}^\infty$ and $\{\mu_{k}\}_{k=0}^\infty$ are  sequences of nonnegative real numbers such that
$$ a_{k+1}\leq(1+\lambda_{k})a_{k}+\mu_{k},\quad k \geq 0.$$
If, in addition, $\sum_{k=0}^\infty \lambda_k<+\infty$ and $\sum_{k=0}^\infty \mu_k<+\infty$, then $\lim_{k\rightarrow \infty} a_{k}$ exists
\end{lemma}

\begin{lemma}{\rm(\cite{Xu2002})}
\label{lem25}
Assume that $\{a_{k}\}_{k=0}^\infty$ is a sequence of nonnegative real numbers such that
$$ a_{k+1}\leq(1-\gamma_{k})a_{k}+\gamma_{k}\delta_{k},~k \geq 0,$$
where $\{\gamma_{k}\}_{k=0}^\infty$ is a sequence in (0,1) and $\{\delta_{k}\}_{k=0}^\infty$ is a real sequence such that
\begin{itemize}
\item[{\rm(i)}]$\sum_{k=0}^{\infty} \gamma_k=\infty$;
\item[{\rm(ii)}]$ \limsup_{k\rightarrow \infty}\delta_{k}\leq 0~ or~ \sum_{k=0}^{\infty} |\gamma_k\delta_{k}|<\infty.$
\end{itemize}
Then $\lim_{k\rightarrow \infty} a_{k}=0.$
\end{lemma}

\section{Halpern Iteration with Adaptive Anchoring Parameters}
\label{Sec:3}

Let $\mathcal{H}$ be a real Hilbert space and let $T: \mathcal{H}\rightarrow \mathcal{H}$ be a nonexpansive mapping with the nonempty fixed point set $Fix(T)$. Consider the Halpern iteration \eqref{Halpern-1}. Here in this section we will introduce a new adaptive way
 to choose the anchoring parameters $\{\lambda_k\}$. Our Halpern iteration (with the anchor and initial guess identical) reads as follows.

\vskip 2mm

\begin{algorithm}(Halpern iteration with adaptive anchoring parameters)\label{Al:3.1}
\rm
\quad

\hspace*{0.1 pc} Step 1: Choose $x^0\in \mathcal{H}$  arbitrarily and set $k:=1$.\\
\hspace*{1.5 pc} Step 2: For the current $x^{k-1}$  $(k\geq 1)$, if $x^{k-1}=Tx^{k-1},$ the iteration process is\\
\hspace*{4.7 pc} terminated. Otherwise (i.e., $x^{k-1}\not=Tx^{k-1}$), calculate
\begin{equation}\label{eq3.2}
x^k=\frac{1}{\varphi_k+1}x^0+\frac{\varphi_k}{\varphi_k+1}Tx^{k-1},
\end{equation}
\hspace*{4.75 pc} where $\{\varphi_k\}$ is given by
  \begin{equation}\label{eq3.1}
 \varphi_k:=\frac{2\langle x^{k-1}-Tx^{k-1}, x^0-x^{k-1}\rangle}{\|x^{k-1}-Tx^{k-1}\|^2}+1.
 \end{equation}
\hspace*{1.5 pc} Step 3:  Set $k:=k+1$ and return to Step 2.
\end{algorithm}

\begin{remark}
\rm
With no loss of generality, we always assume that $x^k\neq Tx^k$ for all $k\geq 0$ in the rest of this section.
Namely, Algorithm \ref{Al:3.1} generates an infinite sequence of iterates $\{x^k\}_{k=0}^\infty$.
\end{remark}

We now discuss the convergence and rate of asymptotic regularity of Algorithm \ref{Al:3.1}.
We begin with establishing some properties of the sequence $\{\varphi_k\}$ coupled with the iterates $\{x^k\}$.

\begin{lemma}\label{lem3.1}
The following properties hold for all $k\geq 1$:
\begin{itemize}
\item[{\rm(i)}] $\varphi_k\geq k$,
\item[{\rm(ii)}] $\|x^k-Tx^k\|^2\leq\frac{2}{\varphi_k}\langle x^k-Tx^k,x^0-x^k\rangle.$
\end{itemize}
\end{lemma}

\begin{proof}
We prove the conclusions (i) and (ii) by induction. For $k=1$, from (\ref{eq3.1}), we get $\varphi_1=1$. Using (\ref{eq3.2}), we have
$$x^1=\frac{1}{2}x^0+\frac{1}{2}Tx^0\,\,\,{\rm or}\,\,\,Tx^0=2x^1-x^0.$$
Consequently,
\begin{align}\label{eq:Tx0}
\|Tx^0-Tx^1\|^2&=\|(x^1-Tx^1)+(x^1-x^0)\|^2 \nonumber\\
&=\|x^1-Tx^1\|^2+\|x^1-x^0\|^2+2\langle x^1-Tx^1, x^1-x^0 \rangle.
\end{align}
By nonexpansiveness of $T$, we have $\|Tx^0-Tx^1\|\le\|x^0-x^1\|$. It turns out from \eqref{eq:Tx0} that
  $$\|x^1-Tx^1\|^2\leq 2\langle x^1-Tx^1, x^0-x^1 \rangle.$$
That is, (ii) holds for $k=1$.

Suppose (i) and (ii) hold for $k-1$ ($k\geq 2$), that is,
\begin{equation}\label{eq3.3}
\varphi_{k-1}\geq k-1;
\end{equation}
\begin{equation}\label{eq3.4}
\|x^{k-1}-Tx^{k-1}\|^2\leq\frac{2}{\varphi_{k-1}}\langle x^{k-1}-Tx^{k-1},x^0-x^{k-1}\rangle.
\end{equation}
By (\ref{eq3.1}), (\ref{eq3.3}) and (\ref{eq3.4}), we get
\begin{equation*}
\aligned
\varphi_k&=\frac{2\langle x^{k-1}-Tx^{k-1}, x^0-x^{k-1}\rangle}{\|x^{k-1}-Tx^{k-1}\|^2}+1\\
&\geq \varphi_{k-1}+1\geq k-1+1=k.
\endaligned
\end{equation*}
Also from (\ref{eq3.2}), we derive that
\begin{equation}\label{eq3.5}
Tx^{k-1}=\frac{\varphi_k+1}{\varphi_k}x^k-\frac{1}{\varphi_k}x^0=x^k+\frac{1}{\varphi_k}(x^k-x^0).
\end{equation}
By nonexpansiveness of $T$ and (\ref{eq3.5}), we have
\begin{align}\label{eq3.6}
\|x^{k-1}-x^k\|^2&\geq \|Tx^{k-1}-Tx^k\|^2=\|(x^k-Tx^k)+\frac{1}{\varphi_k}(x^k-x^0)\|^2 \nonumber\\
&=\|x^k-Tx^k\|^2+\frac{2}{\varphi_k}\langle x^k-Tx^k, x^k-x^0 \rangle+\frac{1}{\varphi_k^2}\|x^k-x^0\|^2.
\end{align}
On the other hand, from (\ref{eq3.2}) again, we obtain
\begin{align}\label{eq3.7}
\|x^{k-1}-x^k\|^2
&=\|(x^{k-1}-Tx^{k-1})-\frac{1}{\varphi_k+1}(x^0-Tx^{k-1})\|^2 \nonumber\\
&=\|x^{k-1}-Tx^{k-1}\|^2-\frac{2}{\varphi_k+1}\langle x^{k-1}-Tx^{k-1}, x^0-Tx^{k-1} \rangle \nonumber\\
&\quad +\frac{1}{(\varphi_k+1)^2}\|x^0-Tx^{k-1}\|^2
\end{align}
and (using (3.6))
\begin{equation}\label{eq3.8}
\frac{1}{\varphi_k^2}\|x^k-x^0\|^2=\frac{1}{(\varphi_k+1)^2}\|x^0-Tx^{k-1}\|^2.
\end{equation}
Combining \eqref{eq3.6}-\eqref{eq3.8}, we obtain
\begin{equation}\label{eq3.9}
\aligned
0&\ge\|Tx^{k-1}-Tx^k\|^2-\|x^{k-1}-x^k\|^2\\
&=\|x^k-Tx^k\|^2-\|x^{k-1}-Tx^{k-1}\|^2+\frac{2}{\varphi_k}\langle x^k-Tx^k, x^k-x^0 \rangle\\
&\quad +\frac{2}{\varphi_k+1}\langle x^{k-1}-Tx^{k-1}, x^0-Tx^{k-1} \rangle.
\endaligned
\end{equation}
On the other hand, it follows from (\ref{eq3.1}) that
\begin{align*}
\varphi_k\|x^{k-1}-Tx^{k-1}\|^2&=2\langle x^{k-1}-Tx^{k-1}, x^0-x^{k-1}\rangle+\|x^{k-1}-Tx^{k-1}\|^2\\
&=2\langle x^{k-1}-Tx^{k-1}, x^0-Tx^{k-1}+Tx^{k-1}-x^{k-1}\rangle+\|x^{k-1}-Tx^{k-1}\|^2\\
&=2\langle x^{k-1}-Tx^{k-1}, x^0-Tx^{k-1}\rangle-\|x^{k-1}-Tx^{k-1}\|^2.
\end{align*}
This implies that
\begin{equation}\label{eq3.10}
\|x^{k-1}-Tx^{k-1}\|^2=\frac{2}{\varphi_k+1}\langle x^{k-1}-Tx^{k-1}, x^0-Tx^{k-1}\rangle.
\end{equation}
Substituting (\ref{eq3.10}) into (\ref{eq3.9}) yields
\begin{equation*}
0\ge\|x^k-Tx^k\|^2+\frac{2}{\varphi_k}\langle x^k-Tx^k, x^k-x^0 \rangle.
\end{equation*}
This is (ii) at $k$, and the proof is finished.
\end{proof}

\subsection{Convergence analysis}

We  are now in the position to prove the strong convergence of Algorithm \ref{Al:3.1}.

\begin{theorem}\label{th3.1}
Assume $\mathcal{H}$ is a real Hilbert space and $T: \mathcal{H}\to \mathcal{H}$ a nonexpansive mapping such that $Fix(T)\not=\emptyset$.
Let  $ \{x^k\}_{k=0}^\infty$ be a sequence generated by Algorithm \ref{Al:3.1}. Then $\{x^k\}_{k=0}^\infty$ converges strongly to
a fixed point of $T$.
\end{theorem}
\begin{proof}
We first prove that $\{x^k\}_{k=0}^\infty$ is bounded. Indeed, for any $p\in Fix(T)$, we have from (\ref{eq3.2}) that

\begin{equation}\label{bounded1}
\aligned
\|x^{k}-p\|=&\left\|\frac{1}{\varphi_k+1}(x^0-p)+\frac{\varphi_k}{\varphi_k+1}(Tx^{k-1}-p)\right\|\\
\leq &\frac{1}{\varphi_k+1}\|x^0-p\|+\frac{\varphi_k}{\varphi_k+1}\|Tx^{k-1}-p\|\\
\leq &\frac{1}{\varphi_k+1}\|x^0-p\|+\frac{\varphi_k}{\varphi_k+1}\|x^{k-1}-p\|\\
\leq &\max\{\|x^0-p\|,\|x^{k-1}-p\|\}.
\endaligned
\end{equation}
By induction, we get
\begin{equation*}
\|x^{k}-p\|\leq \|x^0-p\|\quad \mbox{for all $k\ge 0$}.
\end{equation*}
This means that $\{x^k\}_{k=0}^\infty$ is bounded and hence $\omega_w(x^k)\neq \emptyset$. On the other hand, it follows from Lemma \ref{lem3.1} that
\begin{equation}\label{bounded2}
\|x^k-Tx^k\|\leq \frac{2}{\varphi_k}\|x^0-x^k\|\leq \frac{2}{k}\|x^0-x^k\|
\end{equation}
for all $k\ge 1$. Consequently, the boundedness of $\{x^k\}$ ensures that $\|x^k-Tx^k\|\rightarrow 0$ as $k\rightarrow \infty$.
By Lemma \ref{lemmaGK}, we assert that $\omega_w(x^k)\subset Fix(T)$.

Next, we distinguish two cases.

Case 1: $\sum_{k=0}^\infty \frac{1}{\varphi_k+1}=\infty$.

In this case, we shall prove that $x^k\rightarrow q:=P_{Fix(T)}x^0$, i.e.,
the fixed point of $T$ which is closest from $Fix(T)$ to $x^0$.
To see this we use the definition (\ref{eq3.2}) of $x^k$ to deduce that
\begin{align}\label{case1-1}
\|x^k-q\|^2=&\|\frac{1}{\varphi_k+1}(x^0-q)+\frac{\varphi_k}{\varphi_k+1}(Tx^{k-1}-q)\|^2 \nonumber\\
= &\frac{1}{(\varphi_k+1)^2}\|x^0-q\|^2+\frac{\varphi_k^2}{(\varphi_k+1)^2}\|Tx^{k-1}-q\|^2 \nonumber\\
&+\frac{2\varphi_k}{(\varphi_k+1)^2}\langle x^0-q, Tx^{k-1}-q\rangle \nonumber\\
\leq & (1-\frac{1}{\varphi_k+1})\|x^{k-1}-q\|^2+\frac{1}{(\varphi_k+1)^2}\|x^0-q\|^2 \nonumber\\
&+\frac{2\varphi_k}{(\varphi_k+1)^2}\langle x^0-q, x^{k-1}-q\rangle \nonumber\\
&+\frac{2\varphi_k}{(\varphi_k+1)^2}\|x^0-q\|\|x^{k-1}-Tx^{k-1}\|.
\end{align}
We rewrite \eqref{case1-1} in a more compact form as follows:
\begin{equation}\label{eq:case-1b}
a_{k+1}\le (1-\gamma_k)a_k+\gamma_k\delta_k
\end{equation}
where $a_{k+1}=\|x^k-q\|^2$, $\gamma_k=\frac{1}{\varphi_k+1}$, and
$$\delta_k=\frac{1}{\varphi_k+1}\|x^0-q\|^2
+\frac{2\varphi_k}{\varphi_k+1}(\langle x^0-q, x^{k-1}-q\rangle+\|x^0-q\|\|x^{k-1}-Tx^{k-1}\|).$$
We then have $\gamma_k\to 0$ as $k\to\infty$ and $\sum_{k=1}^\infty\gamma_k=\infty$, due to the assumption of Case 1.

Noticing the fact $\omega_w(x^k)\subset Fix(T)$ and by Lemma \ref{lem22} (recalling that $q=P_{Fix(T)}x^0$), we have
\begin{equation}\label{eq:limsup}
\limsup_{k\rightarrow \infty}\langle x^0-q, x^{k-1}-q\rangle\leq \sup_{p\in \omega_w(x^k)}\langle x^0-q, p-q\rangle\leq 0.
\end{equation}
The facts $\varphi_k\to\infty$ and $\|x^k-Tx^k\|\to 0$ together with \eqref{eq:limsup} readily imply
$\limsup_{k\to\infty}\delta_k\le 0$. Hence, Lemma \ref{lem25} is applicable to \eqref{eq:case-1b}
to get $a_k\to 0$, i.e., $\|x^k-q\|\to 0$.
This finishes the proof of Case 1.

Case 2: $\sum_{k=0}^\infty \frac{1}{\varphi_k+1}<\infty$.

 To handle this case we again set $\gamma_k= \frac{1}{\varphi_k+1}$ for all $k$; thus, $\sum_{k=0}^\infty \gamma_k<\infty$.
Now take $p\in Fix(T)$ to derive from the definition (\ref{eq3.2}) that
\begin{equation}\label{case2-1}
\aligned
\|x^k-p\|&\leq \frac{1}{\varphi_k+1}\|x^0-p\|+\frac{\varphi_k}{\varphi_k+1}\|Tx^{k-1}-p\|\\
&\leq \|x^{k-1}-p\|+\gamma_k\|x^0-p\|.
\endaligned
\end{equation}
Applying Lemma \ref{lem-2} to (\ref{case2-1}) yields that $\lim_{k\rightarrow \infty}\|x^k-p\|$ exists.
This and the fact $\omega_w(x^k)\subset Fix(T)$ make Lemma \ref{lem-1} applicable and we conclude that $\{x^k\}_{k=0}^\infty$ converges weakly  to a
fixed point of $T$.

To prove the strong convergence of $\{x^k\}_{k=0}^\infty$, it suffices to show that $\{x^k\}_{k=0}^\infty$ is a Cauchy sequence.
We rewrite the first inequality in \eqref{bounded2} in terms of $\gamma_k$ as
$$\|x^k-Tx^k\|\le \frac{2\gamma_k}{1-\gamma_k}\|x^0-x^k\|\le 4\gamma_k\|x^0-x^k\|$$
for all $k\ge 1$ since $\gamma_k\le\frac{1}{k+1}\le\frac12$ for all $k\ge 1$.
It turns out that
\begin{equation}\label{case2-2}
\sum_{k=0}^\infty\|x^k-Tx^k\|<\infty.
\end{equation}
On the other hand, from (\ref{eq3.2}), we get
\begin{equation}\label{case2-3}
\aligned
\|x^k-x^{k-1}\|&=\|\frac{1}{\varphi_k+1}(x^0-x^{k-1})+\frac{\varphi_k}{\varphi_k+1}(Tx^{k-1}-x^{k-1})\|\\
&\leq \gamma_k\|x^0-x^{k-1}\|+\|x^{k-1}-Tx^{k-1}\|.
\endaligned
\end{equation}
Hence (\ref{case2-3}) together with (\ref{case2-2}), the fact $\sum_{k=1}^\infty\gamma_k<\infty$ and the boundedness of $\{x_k\}$ leads to
$$\sum_{k=0}^\infty\|x^k-x^{k-1}\|<\infty.$$
This proves that $\{x^k\}_{k=0}^\infty$ is a Cauchy sequence, and the proof of Case 2 is ended.
\end{proof}

\begin{remark}
As pointed out by Halpern \cite{Halpern1967}, in the Halpern iteration  \eqref{Halpern-1},
if the anchoring parameters $\{\lambda_k\}$ are chosen in an open loop way, then the divergence
condition (C2), i.e., $\sum_{k=1}^\infty\lambda_k=\infty$, is a necessary condition for convergence of the iterates $\{x^k\}$.
However, our result in Theorem \ref{th3.1} shows that (C2) is no longer necessary for the convergence of the Halpern iterates
$\{x^k\}$ in the situation where the anchoring parameters $\{\lambda_k\}$ are chosen in an adaptive way.

An advantage of the open loop manner is perhaps that the divergence condition (C2) forces the Halpern iterates $\{x^k\}$
converge to $P_{Fix(T)}u$. Namely, the limit of the iterates $\{x^k\}$ can be identified as the nearest point projection
of the anchor $u$ onto the fixed point set $Fix(T)$.
A natural question thus arisen for the Halpern Algorithm \ref{Al:3.1} is this:
would it be possible that $\sum_{k=0}^\infty \frac{1}{\varphi_k+1}<\infty$ and the iterates $\{x^k\}_{k=0}^\infty$ converge
to a fixed point of $T$ different from $P_{Fix(T)}u$? The following example provides an affirmative answer to this question.
However it remains an interesting problem of how to identify the limit of the iterates $\{x^k\}$ of Algorithm \ref{Al:3.1}
in the case where $\sum_{k=0}^\infty \frac{1}{\varphi_k+1}<\infty$. We will discuss this problem partially in Section \ref{Sec:4}.
\end{remark}
\begin{example}
Let $\mathcal{H}=\mathbb{R}^2$, $D=\{(\xi,\eta)\in\mathbb{R}^2\,|\,\xi+\eta\geq 2\}$ and $H=\{(\xi,\eta)\in\mathbb{R}^2\,|\,\eta=2\}$.
Let $P_H$ and $P_D$ be the projections onto
$H$ and $D$, respectively. It is not hard to find that, for each point $(\xi',\eta')\in \mathbb{R}^2$,
$P_H(\xi',\eta')=(\xi',2)$ and
\begin{equation*}
P_D(\xi',\eta')=\left\{\begin{array}{ll}
(\xi',\eta')&\quad \mbox{if $\xi'+\eta'\ge 2$},\\
(1-\frac12(\eta'-\xi'),1+\frac12(\eta'-\xi'))&\quad \mbox{if $\xi'+\eta'<2$.}\end{array}\right.
\end{equation*}
Now define $T:=P_HP_D$. Then $T$ is nonexpansive and its fixed point set
$Fix(T)=\{(\xi,\eta)\in\mathbb{R}^2\,|\,\xi\geq 0,\,\eta=2\}$.

Taking the initial guess $x^0\equiv (\xi_0,\eta_0)^\top=(0,0)^\top$, we now calculate the iterative sequence $\{x^k\}_{k=1}^\infty$ generated by Algorithm \ref{Al:3.1}. Set $x^{k}=(\xi_k,\eta_k)^\top$ for $k\geq 1$. From (\ref{eq3.1}), we have
\begin{equation}\label{canshu}
\frac{\varphi_k}{1+\varphi_k}=\frac{\|Tx^{k-1}\|^2-\|x^{k-1}\|^2}{2(\|Tx^{k-1}\|^2-\langle x^{k-1},T
x^{k-1}\rangle)}, \,\,\,k\geq 1.
\end{equation}

By using (\ref{canshu}),  we get
\begin{align*}
x^1&=\frac{\varphi_1}{1+\varphi_1}Tx^0=\frac{1}{2}P_HP_Dx^0=\frac{1}{2}P_H(1,1)^\top=\frac{1}{2}(1,2)^\top=(\frac{1}{2},1)^\top,\\
Tx^1&=P_HP_Dx^1=P_H(\frac{3}{4},\frac{5}{4})^\top=(\frac{3}{4},2)^\top,\,\,\,\frac{\varphi_2}{1+\varphi_2}=\frac{53}{70},\ {\rm and} \\
x^2&=\frac{\varphi_2}{1+\varphi_2}Tx^1=\frac{53}{70}(\frac{3}{4},2)^\top=(\frac{159}{280},\frac{53}{35})^\top.
\end{align*}
It turns out that $x^2\in D$. By induction, we can easily prove that for all $k\ge 3$,
$x^{k-1}\in D$,
 \begin{equation}\label{canshu-1}
 \frac{\varphi_k}{1+\varphi_k}=\frac{2+\eta_{k-1}}{4},
 \end{equation}
  and
\begin{equation}\label{ditui}
(\xi_k,\eta_k)^\top=\frac{2+\eta_{k-1}}{4}(\xi_{k-1},2)^\top.
\end{equation}
We rewrite (\ref{ditui}) as
\begin{align}
\xi_k&=\frac{2+\eta_{k-1}}{4}\xi_{k-1}, \label{ditui-a}\\
\eta_k&=1+\frac{1}{2}\eta_{k-1}\label{ditui-b}
\end{align}
for $k\ge 3$. Recursively applying \eqref{ditui-b} yields that, for $k\ge 3$,
\begin{align}\label{ditui-c}
\eta_k=\sum_{j=0}^{k-3}\frac{1}{2^j}+\frac{1}{2^{k-2}}\eta_{2}=2\left(1-\frac{1}{2^{k-2}}\right)+\frac{1}{2^{k-2}}\eta_2.
\end{align}
Then \eqref{canshu-1} and \eqref{ditui-a} are reduced to (for $k\ge 3$)
\begin{equation}\label{canshu-1a}
 \frac{\varphi_k}{1+\varphi_k}=1-\frac{1}{2^{k-2}}+\frac{1}{2^{k-1}}\eta_2
 \end{equation}
 and respectively
\begin{equation}\label{ditui-a1}
\xi_k=\left(1-\frac{1}{2^{k-2}}+\frac{1}{2^{k-1}}\eta_2\right)\xi_{k-1}
=\prod_{j=3}^{k}\left(1-\frac{1}{2^{j-2}}+\frac{1}{2^{j-1}}\eta_2\right)\xi_{2}.
\end{equation}

From \eqref{canshu-1a}, we get
\begin{equation}\label{canshu-2b}
\frac{1}{1+\varphi_k}=\frac{1}{2^{k-2}}-\frac{1}{2^{k-1}}\eta_2
=\frac{1}{2^{k-2}}(1-\frac{1}{2}\eta_2),\quad k\geq 3.
\end{equation}
Hence, $\sum_{k=3}^\infty \frac{1}{1+\varphi_k}<\infty$.

By \eqref{ditui-a1} and \eqref{ditui-c}, we assert that the sequence of iterates $x^k=(\xi_k,\eta_k)^\top$
converges to the point $(\xi^*,2)\in Fix(T)$, where
\begin{equation}\label{xi}
\xi^*=\prod_{j=3}^{\infty}\left(1-\frac{1}{2^{j-2}}+\frac{1}{2^{j-1}}\eta_2\right)\xi_{2}.
\end{equation}

It is evident that $\xi^*\in (0,\xi_2)$ as $\eta_2=\frac{53}{35}$ and $\xi_2=\frac{159}{280}<1$. However, since $P_{Fix(T)}x^0=(0,2)^\top$,
we arrive at the conclusion that the limit $(\xi^*, 2)$ of the Halpern iterates $\{x^k\}$
is distinct from the projection of the anchor (which is also the initial point) $x^0$, as opposed to the case
where $\sum_{k=1}^\infty\frac{1}{\varphi_k+1}=\infty$.

\end{example}

\subsection{Rate of asymptotic regularity}

Below we give an estimate of the asymptotic regularity rate of Algorithm \ref{Al:3.1}.
\begin{theorem}\label{th3.2}
Let  $ \{x^k\}_{k=0}^\infty$ be a sequence generated by Algorithm \ref{Al:3.1}. Then the following inequality holds:
\begin{equation}\label{rate}
\|x^k-Tx^k\|\leq \frac{2}{\varphi_k+1}\|x^0-x^*\|,\quad k\geq 1,
\end{equation}
where $\varphi_k$ is given by (\ref{eq3.1}) and  $x^*$ is an arbitrary fixed point of $T$.
\end{theorem}

\begin{proof}
It is easy to verify that the equality:
\begin{align}\label{eq3.11}
&\varphi \|a\|^2+2\langle a, b\rangle+\|c\|^2-\|a+c\|^2 \nonumber\\
&=\frac{\varphi+1}{2}\|a\|^2-\frac{2}{\varphi+1}\|a+c-b\|^2+\frac{2}{\varphi+1}\|a+c-b-\frac{\varphi+1}{2}a\|^2
\end{align}
holds for all $a,b,c\in \mathcal{H}$ and arbitrary positive number $\varphi$. Setting $\varphi:=\varphi_k$,
$a:=x^k-Tx^k$, $b:=x^k-x^0$, and $c:=Tx^k-x^*$, we have $a+c=x^k-x^*$, $a+c-b=x^0-x^*$. Consequently, it follows from (\ref{eq3.11}) that
\begin{align}\label{eq3.12}
&\varphi_k \|x^k-Tx^k\|^2+2\langle x^k-Tx^k, x^k-x^0\rangle+\|Tx^k-x^*\|^2-\|x^k-x^*\|^2 \nonumber\\
&=\frac{\varphi_k+1}{2}\|x^k-Tx^k\|^2-\frac{2}{\varphi_k+1}\|x^0-x^*\|^2
+\frac{2}{\varphi_k+1}\|x^0-x^*-\frac{\varphi_k+1}{2}(x^k-Tx^k)\|^2.
\end{align}
Combining (\ref{eq3.12}) and  Lemma \ref{lem3.1} (ii), we have
$$\frac{\varphi_k+1}{2}\|x^k-Tx^k\|^2-\frac{2}{\varphi_k+1}\|x^0-x^*\|^2\leq 0.$$
It is immediately clear that the estimate (\ref{rate}) follows.
\end{proof}

\begin{remark}
Since, by Lemma \ref{lem3.1}(ii),  $\varphi_k\geq k$ for all $k\geq 1$, we see that (\ref{rate}) is indeed
a further improvement of (\ref{Lieder}).
Moreover, applying Theorem \ref{th3.2} to Example 1, we obtain from (\ref{canshu-2b}) that
\begin{equation}\label{tight}
\|x^k-Tx^k\|\leq \frac{1}{2^{k-3}}(1-\frac{1}{2}\eta_2)\|x^*\|,\quad k\geq 3,
\end{equation}
where $x^*\in Fix(T)$.
 On the other hand,  we have $Tx^k=(\xi_k,2)^\top$. Consequently,
$x^k-Tx^k=(0,\eta_k-2)^\top$ and by \eqref{ditui-c}
$$\|x^k-Tx^k\|=2-\eta_k=\frac{1}{2^{k-3}}(1-\frac{1}{2}\eta_2).
$$
This shows the superiority of the adaptive parameter sequence given by (\ref{eq3.1}).
\end{remark}
\begin{remark}
The rate (\ref{rate}) is tight, which can be shown by Example 3.1 in \cite{Lieder2021}. In fact, by direct calculation,
it is easy to verify that $\varphi_k=k$ for this example.
\end{remark}

\begin{remark}
The study of convergence and asymptotic regularity of Halpern's iteration has recently been extended to
some of its variations such as modified Halpern iteration \cite{KX2005} and
Tikhonov-Mann iteration \cite{BCM2019,CKL2022}. It is worth of mentioning that \cite{CKL2022}
obtained quantitative rate of asymptotic regularity and metastability of Tikhonov-Mann iteration in CAT(0) spaces.
\end{remark}

\section{Characterization of the limit of Adaptive Halpern Iterates}
\label{Sec:4}

Let $C$ be a nonempty closed convex subset of a Hilbert space $\mathcal{H}$ and let
$T: C\to C$ be a nonexpansive mapping with $Fix(T):=\{x\in C: Tx=x\}\not=\emptyset$.
Take $x^0\in C$ and let $\{x^k\}_{k\ge 1}$ be generated by Halpern's algorithm:
\begin{equation}\label{Halpern-C}
 x^k=\lambda_k u+(1-\lambda_k)Tx^{k-1},\quad k\geq 1,
 \end{equation}
where $u\in C$ is an anchor and $\{\lambda_k\}\subset (0,1)$ is a sequence of anchoring parameters.
\begin{theorem}
Assume $\lambda_k\to 0$ and $x^k\to q$ in norm as $k\to\infty$.
\begin{itemize}
\item[(i)] For each $z\in\mathcal{H}$, we have
\begin{equation}\label{eq:xk1}
\lim_{k\to\infty}\frac{\|x^k-z\|^2-\|Tx^{k-1}-z\|^2}{\lambda_k}=2\langle u-q,q-z\rangle.
\end{equation}
\item[(ii)] Set
\begin{equation}\label{eq:H0}
H_0=\left\{z\in \mathcal{H}: \lim_{k\to\infty}\frac{\|x^k-z\|^2-\|Tx^{k-1}-z\|^2}{\lambda_k}=0\right\}
\end{equation}
and
\begin{equation}
H_{+}=\left\{z\in \mathcal{H}: \lim_{k\to\infty}\frac{\|x^k-z\|^2-\|Tx^{k-1}-z\|^2}{\lambda_k}\ge 0\right\}.
\end{equation}
Then
\begin{equation}\label{eq:H01}
q=P_{H_0\cap C}u=P_{H_{+}\cap Fix(T)}u.
\end{equation}
\item[(iii)] If, in addition, $\sum_{k=1}^\infty \lambda_k=\infty$, then $q=P_{H_0\cap C}u=P_{Fix(T)}u$.
\end{itemize}
\end{theorem}
\begin{proof}
First observe that $q\in Fix(T)$. Now for each $z\in \mathcal{H}$, it follows from \eqref{Halpern-C} that
\begin{align*}
\|x^k-z\|^2&=\|\lambda_k (u-z)+(1-\lambda_k)(Tx^{k-1}-z)\|^2\\
&=\lambda_k^2\|u-z\|^2+(1-\lambda_k)^2\|Tx^{k-1}-z\|^2+2\lambda_k(1-\lambda_k)\langle u-z,Tx^{k-1}-z\rangle.
\end{align*}
Since $\lambda_k\to 0$ and $x^k\to q\in Fix(T)$, it turns out that
\begin{align*}
&\lim_{k\to\infty}\frac{\|x^k-z\|^2-\|Tx^{k-1}-z\|^2}{\lambda_k}\\
&=\lim_{k\to\infty}\{\lambda_k\|u-z\|^2-(2-\lambda_k)\|Tx^{k-1}-z\|^2+2(1-\lambda_k)\langle u-z,Tx^{k-1}-z\rangle\}\\
&=-2\|q-z\|^2+2\langle u-z,q-z\rangle=2\langle u-q,q-z\rangle.
\end{align*}
This proves (i).
As a consequence of (i), we can rewrite the sets $H_0$ and $H_{+}$ as
$$H_0=\{z\in \mathcal{H}: \langle u-q,q-z\rangle=0\}\quad {\rm and}\quad
H_{+}=\{z\in \mathcal{H}: \langle u-q,q-z\rangle\ge 0\}.$$
Hence, it is trivial that $q\in H_0\subset H_{+}$ and (\ref{eq:H01}) holds.

(iii) Under the condition $\sum_{k=1}^\infty \lambda_k=\infty$,
it has already been proved that $q=P_{Fix(T)}u$. Thus,
$\langle u-q,q-z\rangle\ge 0$ for all $z\in Fix(T)$; that is, $Fix(T)\subset C\cap H_{+}$.
Moreover, we also find that $q=P_{C\cap H_{+}}u=P_{H_{+}}u$ by definition of $H_{+}$.
\end{proof}

\section{Numerical Experiments}
\label{Sec:5}
From the previous analysis, it can be seen that the rate of asymptotic regularity of Algorithm \ref{Al:3.1} is better than the usual Halpern iteration with $\lambda_k=\frac{1}{k+1}$, especially, in the case where $\sum_{k=1}^\infty \frac{1}{\varphi_k+1}<+\infty$, the acceleration effect of Algorithm \ref{Al:3.1} is much more prominent.
In this section,  to test the  effectiveness of the proposed adaptive algorithm (Algorithm \ref{Al:3.1}),  we  compare Algorithm \ref{Al:3.1} and Halpern iteration with the parameter sequence $\{\frac{1}{k+1}\}_{k=1}^\infty$ through two numerical examples.

\begin{example}\label{eg:5.1}
\label{ex1}
\rm
Consider the fixed point problem in \cite{He2021} for the mapping $T:\mathbb{R}^3 \rightarrow  \mathbb{R}^3$ defined by
$$
T(x,y,z)=
\left(
\aligned
&\frac{-35x-\sqrt{|x|+1}-10y+14z+1}{54.5}\\
&\frac{-10x-26y-\frac{1}{2}\sin(y)+4z}{54.5}\\
&\frac{14x+4y-38z-\arctan(\frac{z}{2})}{54.5}
\endaligned
\right),
\quad  (x,y,z)^{\top} \in \mathbb{R}^3.
$$
It is easy to verify that $T$ is nonexpansive.
\end{example}

\begin{figure}[!h]
\setlength{\floatsep}{0pt} \setlength{\abovecaptionskip}{0pt}
\centering
\scalebox{0.65} {\includegraphics{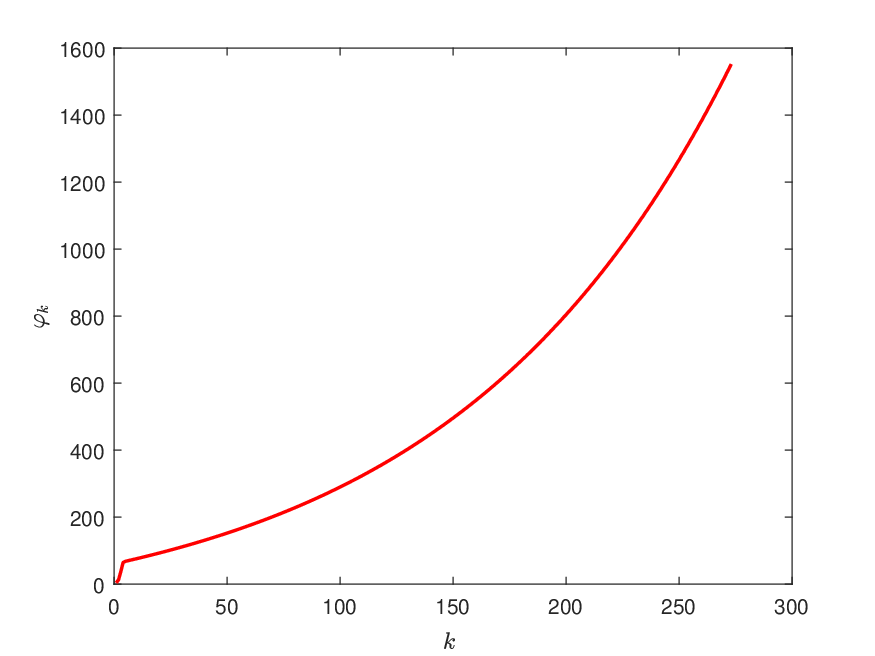}}
\caption{Increase of $\varphi_k$ with $k$ for Example 5.1.}
\label{phi1}
\end{figure}

\begin{figure}[!h]
\setlength{\floatsep}{0pt} \setlength{\abovecaptionskip}{0pt}
\centering
\scalebox{0.65} {\includegraphics{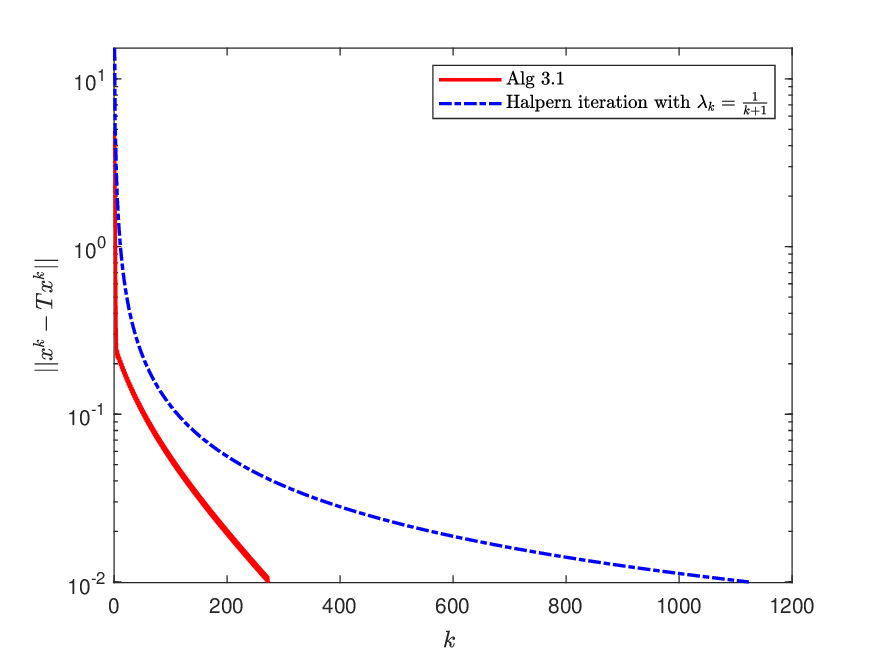}}
\caption{Comparison of Algorithm \ref{Al:3.1} and Halpern iteration with $\lambda_k=\frac1{k+1}$ for Example \ref{eg:5.1}.}
\label{fig1}
\end{figure}

Figure \ref{phi1} illustrates that  $\varphi_k=O(k^2)$,  which belongs to the case where $\sum_{k=1}^\infty \frac{1}{\varphi_k+1}<+\infty$; so it is observed from Figure \ref{fig1} that Algorithm 3.1 behaves much better than Halpern iteration with  $\lambda_k=\frac{1}{k+1}$.

\begin{example}\label{eg:5.2}
\label{ex2}
\rm
 Consider the following LASSO problem \cite{Tibshirani1996}:
\begin{equation}\label{lasso}
\min\frac{1}{2}\|Ax-b\|^{2} +\tau \|x\|_1
\end{equation}
where $A\in  \Bbb R^{m\times n}$, $m< n$, $b\in  \Bbb R^m$ and $\tau >0$. We generate the matrix $A$ from a standard normal distribution with mean zero and unit variance. The true sparse signal $x^*$ is generated from uniform distribution in the interval $[-2,2]$ with random $K$ position nonzero while the rest are kept zero. The sample data $b= Ax^*$.
\end{example}

By the first-order optimality condition of LASSO problem \eqref{lasso}, we have the following fixed point equation
$$
x={\rm prox}_{\gamma\tau\|\cdot\|_1}(x-\gamma A^\top(Ax-b)),
$$
where $\gamma>0$, ``${\rm prox}$" is a proximal operator. It is known that the proximal operator of $\ell_1$-norm is
given componentwise by
$$({\rm prox}_{\gamma\tau\|\cdot\|_1}(x))_i={\rm sign}(x_i)\max\{|x_i|-\gamma\tau,0\}$$
for $x=(x_1,x_2,\cdots,x_n)^{\top}\in\mathbb{R}^n$ and $i=1,2,\cdots,n$.
Let $\gamma\in (0,\frac{2}{\|A\|^2})$, then the mapping $T:\mathbb{R}^n\rightarrow\mathbb{R}^n$ defined by
$$Tx :={\rm prox}_{\gamma\tau\|\cdot\|_1}(x-\gamma A^\top(Ax-b))$$
is nonexpansive (see \cite{Xu2011}).
\vskip 2mm

In the numerical results listed in Table \ref{Tab1}, we consider $(m, n, K) = (120i, 512i, 20i)$ for $i = 1, 2, \ldots, 10$. We run 10 instances randomly for each $(m, n, K)$ and report the number of iterations (Iter), CPU time in seconds (CPU time) and the relative error (Err) defined as
$$
{\rm Err} := \frac{\|\hat{x} - x^*\|}{\|x^*\|},
$$
with $\hat x$ being the recovered sparse solution by algorithms.
We terminate the algorithms in the experiment when
$$
{\|x^k - Tx^{k}\|} < 10^{-4}.
$$

As we can see from Figure \ref{phi2}, the parameter $\varphi_k$ exponentially increases as $k$ increases,  similar to Example \ref{eg:5.1} which also belongs to the case of $\sum_{k=1}^\infty \frac{1}{\varphi_k+1}<+\infty$. Thus from Table \ref{Tab1} and Figure \ref{fig2}, it is observed that Algorithm \ref{Al:3.1} performs much better than Halpern iteration with $\lambda_k=\frac1{k+1}$ in terms of Iter, CPU time and Err.

From the above two examples, we find that the case of $\sum_{k=1}^\infty \frac{1}{\varphi_k+1}<+\infty$ does not just occur in some simple examples like Example \ref{eg:5.1}, but also in practical problems such as the LASSO problem. This phenomenon shows that Algorithm 3.1 has practical computational value.

\begin{table}[h]
\label{table}
\small
\centering
\caption{Computational results of Algorithm 3.1 and Halpern iteration with $\lambda_k=\frac1{k+1}$ for LASSO problem}
\vskip 2mm
\begin{tabular}{cccccccccccccccccc}
\hline
\multicolumn{3}{c}{Problem size}  & \multicolumn{4}{c}{Algorithm 3.1} &   \multicolumn{4}{c}{Halpern iteration with $\lambda_k=\frac1{k+1}$}\\
\cline{1-3} \cline{5-7} \cline{9-11}
$m$ & $n$ & $K$ && Iter & CPU time & Err && Iter & CPU time & Err \\
\hline
 \hline
120&512&20&&4245&1.6219&0.0599&&48256&17.8234&0.0617\\
240&1024&40&&10246&9.0219&0.0459&&73516&61.3859&0.0480\\
360&1536&60&&14243&28.6266&0.0367&&88838&179.0906&0.0384\\
480&2048&80&&17367&57.8031&0.0289&&101470&335.3656&0.0306\\
600&2560&100&&21540&109.1563&0.0268&&113000&560.2719&0.0285\\
720&3072&120&&26315&176.2688&0.0262&&126720&830.5828&0.0280\\
840&3584&140&&32154&269.7188&0.0245&&138460&1179.5&0.0266\\
960&4096&160&&34949&367.6172&0.0219&&147290&1545.2&0.0238\\
1080&4608&180&&38025&503.0031&0.0206&&152360&2035.3&0.0225\\
1200&5120&200&&43812&734.1328&0.0197&&161050&2660.8&0.0216\\
\hline
\hline
\end{tabular}\label{Tab1}
\end{table}

\begin{figure}[!h]
\setlength{\floatsep}{0pt} \setlength{\abovecaptionskip}{0pt}
\centering
\scalebox{0.65} {\includegraphics{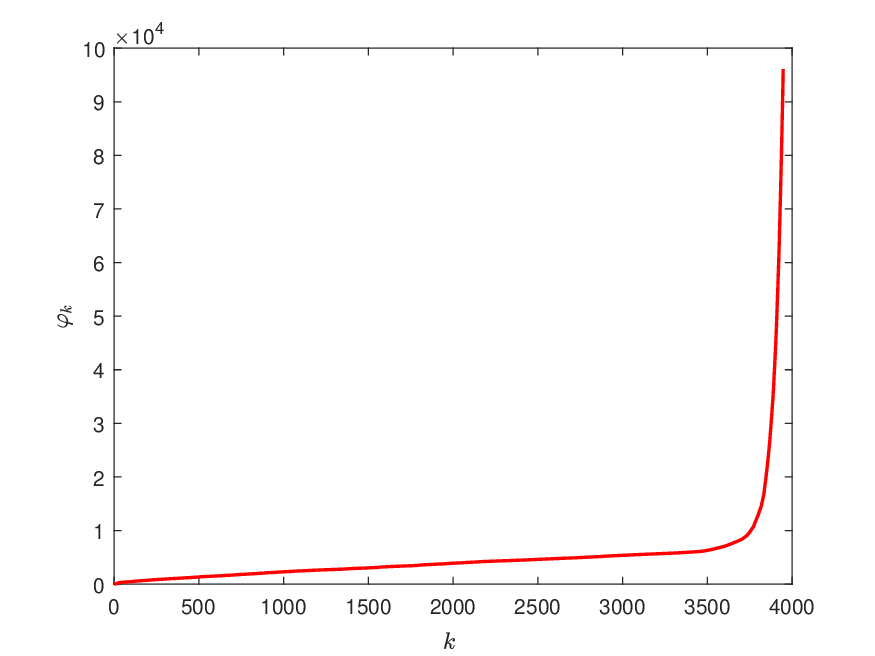}}
\caption{Increase of $\varphi_k$ with $k$ for LASSO problem}
\label{phi2}
\end{figure}

\begin{figure}[!h]
\setlength{\floatsep}{0pt} \setlength{\abovecaptionskip}{0pt}
\centering
\scalebox{0.65} {\includegraphics{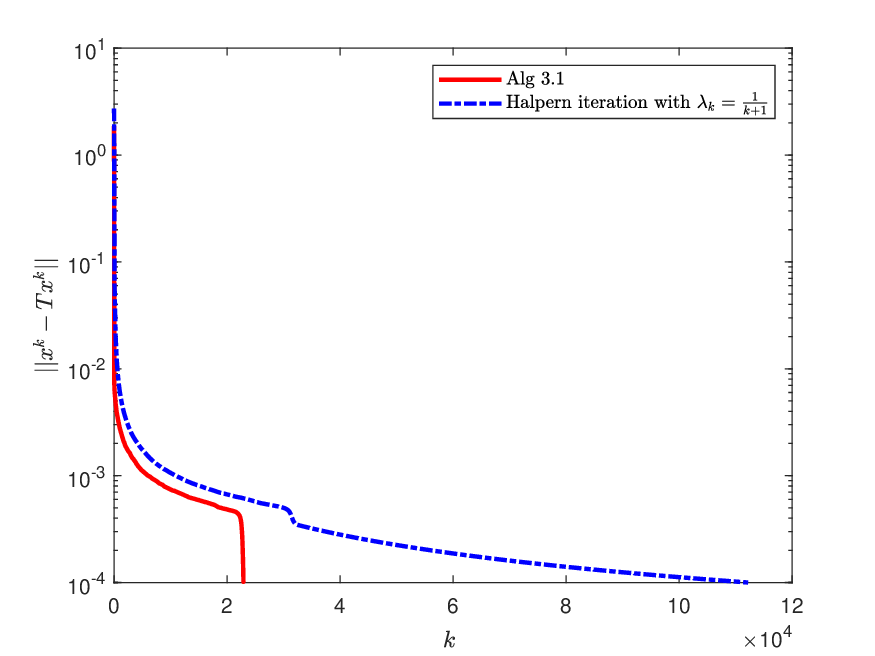}}
\caption{Comparison of Algorithm \ref{Al:3.1} and Halpern iteration with $\lambda_k=\frac1{k+1}$ for Lasso problem with $m=600$, $n=2560$ and $K=100$ for the cameraman.}
\label{fig2}
\end{figure}

\begin{example}\label{eg:5.3}
\label{ex3}
\rm
Consider the cameraman test image \cite{A2009}:
\begin{equation}\label{eg5.3}
\min F(x)\equiv\|Ax-b\|^{2} +\tau \|x\|_1
\end{equation}
 where $b$ represents the
(vectorized) observed image, and $A = RW$, where $R$ is the matrix representing the blur operator and $W$ is the inverse of a three stage Haar wavelet transform. The regularization parameter is chosen to be $\tau = 2e-5$, and the initial image is the blurred image.

All pixels of the original images described in the examples are first scaled into the range between $0$ and $1$. In the example we look at the $256\times256$ cameraman test image. The image goes through a Gaussian blur of size $9\times9$ and standard deviation $4$ followed by an additive zero-mean white Gaussian noise. The original and observed images are given in Figure \ref{C2}.
\end{example}

For these experiments we assume reflexive (Neumann) boundary conditions \cite{H2006}. We then test Algorithm \ref{Al:3.1} and Halpern iteration with $\lambda_k=\frac1{k+1}$ for solving problem \eqref{eg5.3}. As we can see from Figure \ref{phi3}, $\varphi_k=O(k^2)$, similar to Example \ref{eg:5.1}, which also belongs to the case of $\sum_{k=1}^\infty \frac{1}{\varphi_k+1}<+\infty$. Iterations 500 and 1000 are described in Figure \ref{C1}. The function value at iteration $k$ is denoted by $F_k$.  The function value of Algorithm \ref{Al:3.1} is consistently lower than the function values of Halpern iteration. Note that the function value of Halpern iteration with $\lambda_k=\frac1{k+1}$ after 1000 iterations, is still worse (that is, larger) than the function value of Algorithm \ref{Al:3.1} after 500 iterations.

\begin{figure}[!h]
\setlength{\floatsep}{0pt} \setlength{\abovecaptionskip}{0pt}
\centering
\scalebox{0.85} {\includegraphics{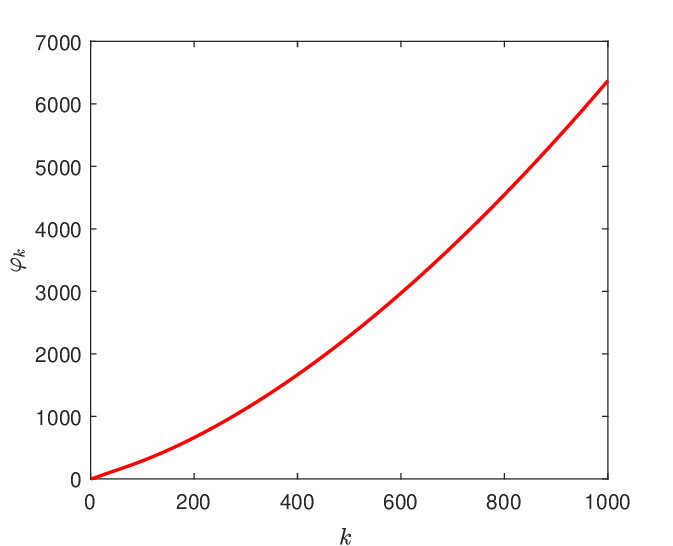}}
\caption{Increase of $\varphi_k$ with $k$ for the cameraman.}
\label{phi3}
\end{figure}

\begin{figure}[!h]
\setlength{\floatsep}{0pt} \setlength{\abovecaptionskip}{0pt}
\centering
\scalebox{0.65} {\includegraphics{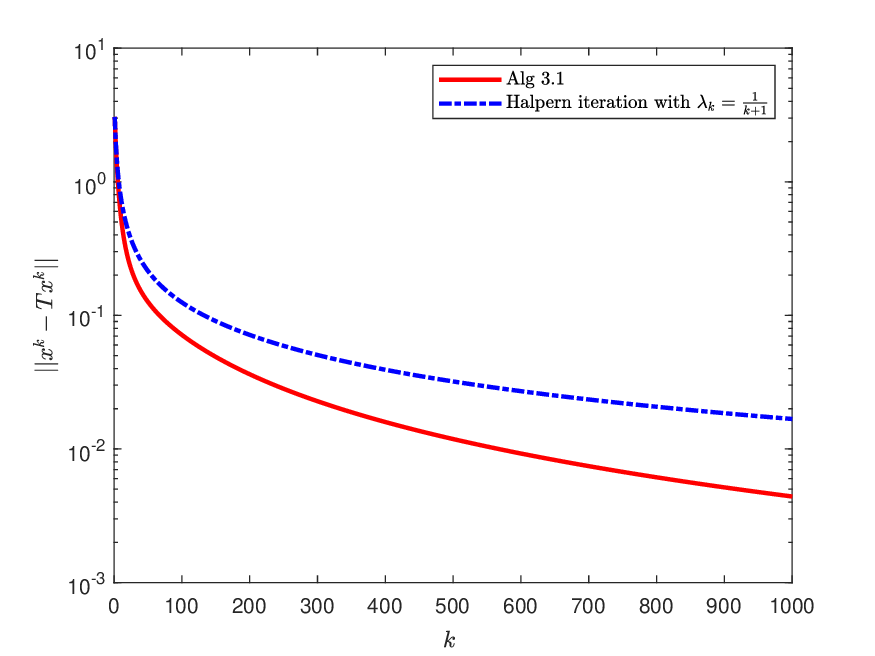}}
\caption{Comparison of Algorithm 3.1 and Halpern iteration with $\lambda_k=\frac1{k+1}$ for the cameraman.}
\label{fig1}
\end{figure}

\begin{figure}[htbp]
\centering
 \begin{minipage}{0.49\linewidth}
  \centering
  \includegraphics[width=1.0\linewidth]{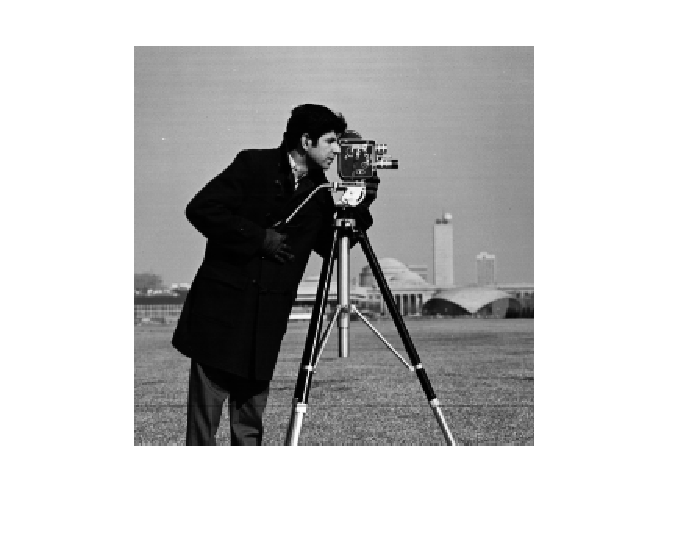}
  \caption*{original}
 \end{minipage}
 \begin{minipage}{0.49\linewidth}
  \centering
  \includegraphics[width=1.0\linewidth]{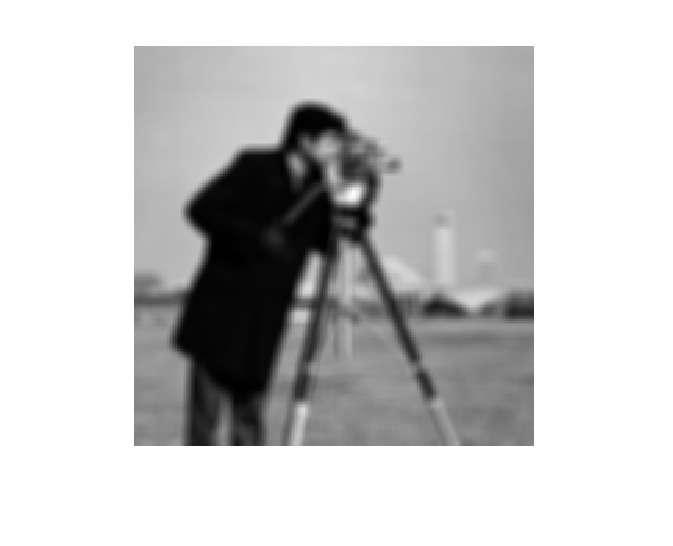}
  \caption*{blurred and noisy}
 \end{minipage}
\caption{Deblurring of the cameraman.}
\label{C2}

\vskip 3mm
 \centering
 \begin{minipage}{0.49\linewidth}
  \centering
  \includegraphics[width=1.0\linewidth]{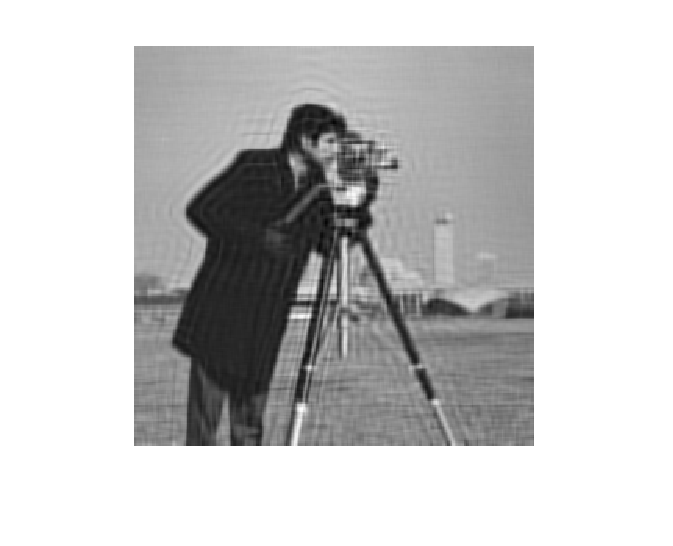}
  \caption*{Alg 3.1: $F_{500}=0.35245$}
 \end{minipage}
 \begin{minipage}{0.49\linewidth}
  \centering
  \includegraphics[width=1.0\linewidth]{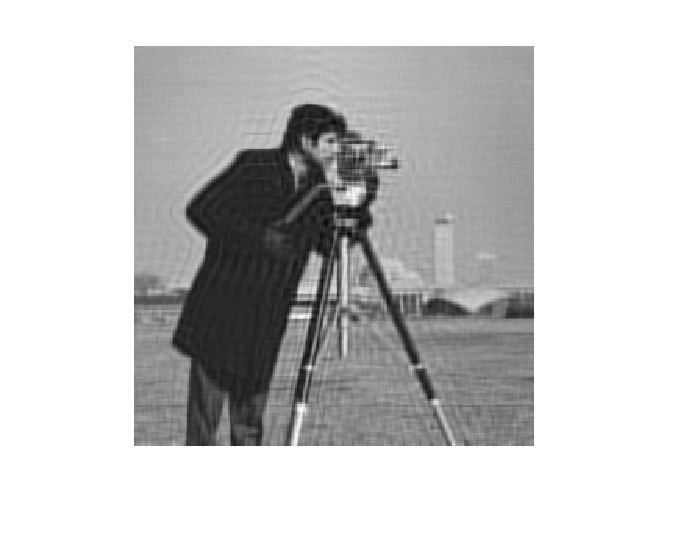}
  \caption*{Alg 3.1: $F_{1000}=0.32412$}
 \end{minipage}

 \centering
 \begin{minipage}{0.49\linewidth}
  \centering
  \includegraphics[width=1.0\linewidth]{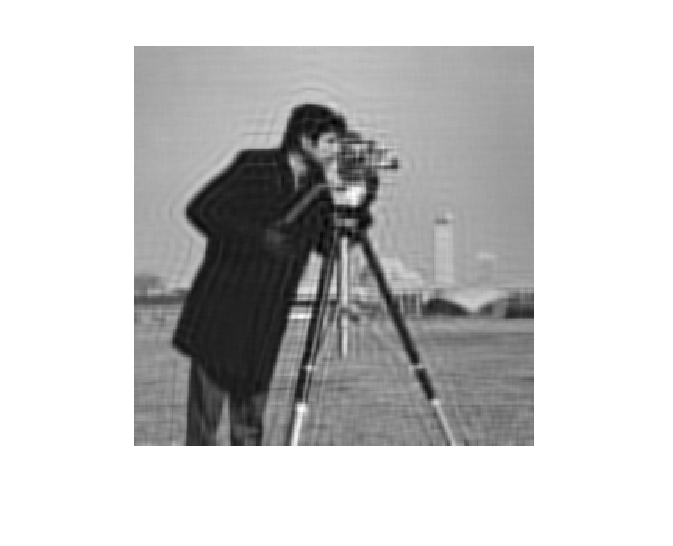}
  \caption*{Halpern iteration: $F_{500}=0.42485$}
 \end{minipage}
 \begin{minipage}{0.49\linewidth}
  \centering
  \includegraphics[width=1.0\linewidth]{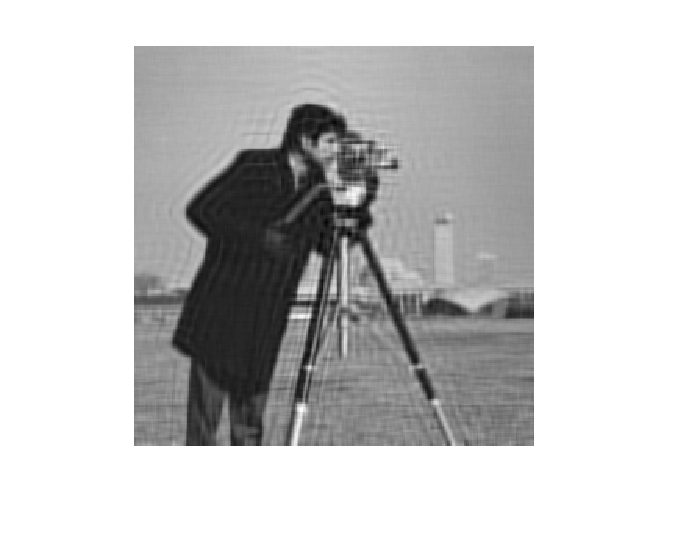}
  \caption*{Halpern iteration: $F_{1000}=0.35463$}
 \end{minipage}
 \caption{Iterations of Algorithm 3.1 and Halpern iteration with $\lambda_k=\frac1{k+1}$ for deblurring of the cameraman.}
 \label{C1}
 \end{figure}

\section{Conclusion}

We have studied the Halpern iteration in the case where the anchoring parameters are chosen in an adaptive manner
to improve the case where the anchoring parameters are chosen in an open loop way.
We have proved strong convergence of this method and obtained the rate of asymptotic regularity at least $O(1/k)$.
Our numerical experiments have shown that our adaptive Halpern iteration  outperforms the ordinary
Halpern iteration.

\vskip 4mm

\section*{Acknowledgements}
We were deeply grateful to the two anonymous referees for their constructive suggestions and critical comments on the manuscript, which
helped us significantly improve the presentation of this paper.
\vskip 3mm

\noindent{\bf Funding}\\
This work was supported by the Open Fund of Tianjin Key Lab for Advanced Signal
Processing (2022ASP-TJ01).
 Xu was supported in part by National Natural Science Foundation of China (grant number U1811461) and by Australian Research Council
 (grant number DP200100124). Dong was supported in part by National Natural Science Foundation of China (Grant No. 1227127).

\vskip 2mm

\noindent{\bf Competing interests}\\
 The authors declare that they have no competing interests.
\vskip 2.0mm

\end{document}